\documentclass[11pt,a4paper, fullpage]{amsart}

\usepackage[centertags]{amsmath}
\usepackage{latexsym,amsthm,amssymb,verbatim}
\usepackage{epsf}
\usepackage[scanall]{psfrag}
\usepackage{graphicx, color}

\newtheorem{defin}{Definition}[section]
\newtheorem{prop-def}[defin]{Proposition-Definition}
\newtheorem{lem}[defin]{Lemma}
\newtheorem{thm}[defin]{Theorem}
\newtheorem{remark}[defin]{Remark}
\newtheorem{cor}[defin]{Corollary}

\newtheorem{ex}[defin]{Example}

\newtheorem{conv}[defin]{Convention}
\newcommand{\e}{ \hfill $\diamond$}

\begin{document}

\title[Kurosh Decomposition]{Reading Off Kurosh Decompositions}%
\author{L.Markus-Epstein}\footnote{Supported in part at the Technion by a fellowship of the Israel Council for Higher Education}%
\address{Department of Mathematics \\
Technion \\
Haifa 32000, Israel}%
\email{epstin@math.biu.ac.il}%


\begin{abstract}

Geometric methods proposed by Stallings \cite{stal} for treating
finitely generated subgroups of free groups  were successfully
used to solve a wide collection of decision problems for free
groups and their subgroups \cite{b-m-m-w, kap-m, mar_meak, m-s-w,
mvw, rvw, ventura}.

In the present paper we employ the generalized Stallings' folding
method  developed  in \cite{m-foldings} to introduce a  procedure,
which given a subgroup $H$ of a free product of finite groups
reads off its Kurosh decomposition from the subgroup graph of $H$.

\end{abstract}
\maketitle


\section{Introduction}
\label{sec:Introduction}

The celebrated theorem of Kurosh describes subgroups of free
products.
\begin{thm}[Kurosh Subgroup Theorem \cite{l_s}]
Let $G$ be a free product of groups $G_i$, where $i$ runs over an
index set $I$. Let $H$ be a subgroup of $G$. Then $H=F \ast ( \ast
g_jH_jg_j^{-1})$ is a free product of a free group $F$ together
with groups that are conjugates of subgroups $H_j$ of the free
factors $G_i$ of $G$.
\end{thm}

In this issue one can ask the following algorithmic question.
\emph{Given a subgroup $H$} (for instance, by a finite set of
generators) \emph{of a free product $G=\ast G_i$, find its Kurosh
decomposition $H=F \ast ( \ast g_jH_jg_j^{-1})$ efficiently}.

Below we solve this algorithmic problem (we call it the
\emph{Kurosh decomposition problem})  for finitely generated
subgroups of free products of finite groups, employing graph
theoretical methods developed by the author in \cite{m-foldings}.
More precisely, we introduce an algorithm which \emph{reads off}
the decomposition of a subgroup from its subgroup graph.

This approach goes back to the remarkable paper of Stallings
\cite{stal}, where finitely generated subgroups of free groups
were canonically represented by finite labelled graphs. Later on
this method was successfully applied to solve various algorithmic
problems in free groups \cite{b-m-m-w, kap-m, mar_meak, m-s-w,
mvw, rvw, ventura}, providing mostly polynomial algorithms.

In \cite{m-foldings} Stallings method, or so called
\emph{Stallings' folding algorithm}, was generalized to the class
of amalgams of finite groups. We refer to this generalized
algorithm as the \emph{generalized Stallings' folding algorithm}.
In the current paper our methods are restricted to the case of
free products of finite groups.  The description  of the
generalized Stallings' algorithm (restricted to the case of free
products of finite groups) is included in the Appendix.

Note that the graph constructed by Stallings' folding algorithm
for $S \leq FG(X)$ is the Geodesic core of the coset Cayley graph
of $FG(X)$ relative to $S$, that is the the union of all closed
geodesic paths starting at the basepoint $S \cdot 1$. The
resulting graph $(\Gamma(H),v_0)$ constructed by the generalized
Stallings' folding algorithm for $H \leq G$ is a sort of a
\emph{core graph} as well (see \cite{m-foldings} for more
details). More precisely, it is the \emph{Normal core} of  the
coset Cayley graph of $G$ relative to $H$: the union of all closed
normal paths starting at the basepoint $H \cdot 1$.
 Another example
of core construction can be found in \cite{c_turner}, where
Collins and Turner use a topological approach to study
automorphisms of free products.

The paper is organized as follows. We start (Section
\ref{sec:Preliminaries}) by fixing the notation and by brief
recalling of some known results which are essential for the
current paper.   Readers familiar with free products, normal
(reduced) words  and labelled graphs can skip it. The next section
(Section \ref{sec:SubgroupGraphs}) presents a summary of the
results from \cite{m-foldings} concerning subgroup graphs which
are essential for the solution of Kurosh decomposition problem.

Section \ref{subsection: TheBasicStep} presents  the basic step of
our ``reading'' procedure described along with the proof of
Theorem \ref{cor: alg_Kurosh} (Section \ref{subsection: The
Reading off Kurosh Decomposition}). The complexity analysis of
this algorithm shows that it is quadratic in the size of the
input. The algorithm application is demonstrated in Example
\ref{ex:KuroshDecomp} (Section \ref{subsection: The Reading off
Kurosh Decomposition}).


\section{Acknowledgments}

I wish to deeply thank to my PhD advisor Prof. Stuart W. Margolis
for introducing me to this subject, for his help  and
encouragement throughout  my work on the thesis. I owe gratitude
to Prof. Arye Juhasz for his suggestions and many useful comments
during the writing of this paper. I gratefully acknowledge a
partial support at the Technion by a fellowship of the Israel
Council for Higher Education.


\section{Preliminaries}
\label{sec:Preliminaries}

\subsection*{Free Products}
Throughout this paper, we assume that $G=G_1 \ast G_2$ is a free
product of finite groups $G_1$ and $G_2$ where
\begin{align}
G_1=gp\langle X_1|R_1\rangle, \ \ G_2=gp\langle X_2|R_2\rangle \ \
{\rm such \ that} \ \ X_1^{\pm} \cap X_2^{\pm}=\emptyset.
\tag{\text{$\ast$}}
\end{align}
 Thus
\begin{align}
{ G=gp\langle X_1,X_2 | R_1, R_2 \rangle. } \tag{\text{$\ast
\ast$}}
\end{align}
We denote $X=X_1 \cup X_2$ and put $H$ to be a finitely generated
subgroup of $G$.

Elements of $G=gp \langle X |R \rangle$ are equivalence classes of
words. However it is customary to blur the distinction between a
word $u$ and the equivalence class containing $u$. We will
distinguish between  them by using different equality signs:
\fbox{``$\equiv$''} for the equality of two words and
\fbox{``$=_G$''} to denote the equality of two elements of $G$,
that is the equality of two equivalence classes.


\subsection*{Normal Forms} Let $G=G_1 \ast_{A} G_2$.

A word $g_1g_2 \cdots g_n \in G$ ($n \geq 0$) is in \emph{normal
form} (or, more customary, it is a \emph{normal word}) if the
following holds
\begin{enumerate}
    \item [(1)] $g_i \neq_G 1$ lies in either $G_1$ or $G_2$,
    \item [(2)] $g_i$ and $g_{i+1}$ are in different factors of $G$,
\end{enumerate}
We call the sequence $(g_1, g_2, \ldots, g_n)$ a
 \emph{normal decomposition} of the element $g \in G$, where $g=_G g_1g_2 \cdots g_n$.

By the Normal Form Theorem for Free Products (Theorem IV.1.2 in
\cite{l_s}), the number $n$ is uniquely determined for a given
element $g$ of $G$ and it is called the {\it syllable length} of
$g$.


\subsection*{Labelled graphs}
Below we follow the notation of \cite{gi_sep, stal}.

A graph $\Gamma$ consists of two sets $E(\Gamma)$ and $V(\Gamma)$,
and two functions $E(\Gamma)\rightarrow E(\Gamma)$  and
$E(\Gamma)\rightarrow V(\Gamma)$: for each $e \in E$ there is an
element $\overline{e} \in E(\Gamma)$ and an element $\iota(e) \in
V(\Gamma)$, such that $\overline{\overline{e}}=e$ and
$\overline{e} \neq e$.

The elements of $E(\Gamma)$ are called \textit{edges}, and an $e
\in E(\Gamma)$ is a \emph{direct edge} of $\Gamma$, $\overline{e}$
is the \emph{reverse (inverse) edge} of $e$.

The elements of $V(\Gamma)$ are called \textit{vertices},
$\iota(e)$ is the \emph{initial vertex} of $e$, and
$\tau(e)=\iota(\overline{e})$ is the \emph{terminal vertex} of
$e$. We call them the \emph{endpoints} of the edge $e$.

A  \emph{path of length $n$} is  a sequence of $n$ edges $p=e_1
\cdots  e_n $ such that $v_i=\tau(e_i)=\iota(e_{i+1})$ ($1 \leq
i<n$).  We call $p$ a \emph{path from $v_0=\iota(e_1)$ to
$v_n=\tau(e_n)$}. The \emph{inverse} of the path $p$ is
$\overline{p}=\overline{e_n} \cdots \overline{e_1}$. A path of
length 0 is the \emph{empty path}.

We say that the graph $\Gamma$ is \emph{connected} if $V(\Gamma)
\neq \emptyset$ and any two vertices  are joined by a path. The
path $p$ is \emph{closed} if $\iota(p)=\tau(p)$, and it is
\emph{freely reduced} if $e_{i+1} \neq \overline{e_i}$ ($1 \leq i
<n$). $\Gamma$ is a \emph{tree} if it is a connected graph and
every closed freely reduced path in $\Gamma$ is empty.

A \emph{subgraph} of $\Gamma$ is a graph $C$  such that $V(C)
\subseteq V(\Gamma)$ and $E(C) \subseteq E(\Gamma)$. In this case,
by abuse of language, we write $C\subseteq \Gamma$.
Similarly, whenever we write $\Gamma_1 \cup \Gamma_2$ or $\Gamma_1
\cap \Gamma_2$,  we always mean that the set operations are, in
fact,  applied to the vertex sets and the edge sets of the
corresponding graphs.

A \emph{labelling} of $\Gamma$ by the set $X^{\pm}$ is a function
$$lab: \: E(\Gamma)\rightarrow X^{\pm}$$ such that for each $e \in
E(\Gamma)$, $lab(\overline{e}) \equiv (lab(e))^{-1}$.

The last equality enables one, when representing the labelled
graph $\Gamma$ as a directed diagram,  to represent only
$X$-labelled edges, because $X^{-1}$-labelled edges can be deduced
immediately from them.

A graph with a labelling function is called a \emph{labelled (with
$X^{\pm}$) graph}.  The only graphs considered in the present
paper are labelled graphs.

A labelled graph is called \emph{well-labelled} if
$$\iota(e_1)=\iota(e_2), \; lab(e_1) \equiv lab(e_2)\ \Rightarrow \
e_1=e_2,$$ for each pair of edges $e_1, e_2 \in E(\Gamma)$. See
Figure \ref{fig: labelled, well-labelled graphs}.

\begin{figure}[!h]
\psfrag{a }[][]{$a$} \psfrag{b }[][]{$b$} \psfrag{c }[][]{$c$}
\psfrag{e }[][]{$e_1$}
\psfrag{f }[][]{$e_2$}
\psfragscanon \psfrag{G }[][]{{\Large $\Gamma_1$}}
\psfragscanon \psfrag{H }[][]{{\Large $\Gamma_2$}}
\psfragscanon \psfrag{K }[][]{{\Large $\Gamma_3$}}
\includegraphics[width=\textwidth]{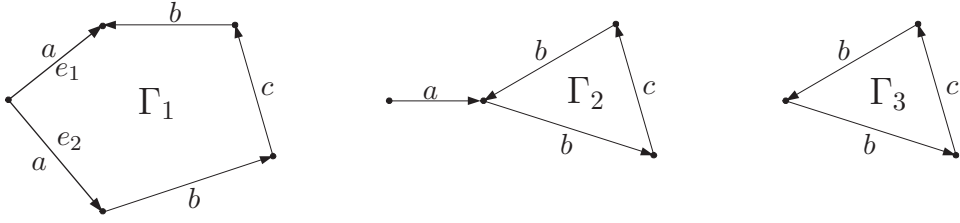}
\caption[The construction of $\Gamma(H_1)$]{ \footnotesize {The
graph $\Gamma_1$ is labelled with $\{a,b,c\} ^{\pm}$, but it is
not well-labelled. The graphs $\Gamma_2$ and $\Gamma_3$ are
well-labelled with $\{a,b,c\} ^{\pm}$.}
 \label{fig: labelled, well-labelled graphs}}
\end{figure}

If a finite graph $\Gamma$ is not well-labelled then a process of
iterative identifications of each pair  $\{e_1,e_2\}$ of distinct
edges with the same initial vertex and the same label to a single
edge yields a well-labelled graph. Such identifications are called
\emph{foldings}, and the whole process is known as the process of
\emph{Stallings' foldings} \cite{b-m-m-w, kap-m, mar_meak, m-s-w}.
 Thus the graph $\Gamma_2$ on Figure
\ref{fig: labelled, well-labelled graphs}  is obtained from the
graph $\Gamma_1$ by folding the edges $e_1$ and $e_2$ to a single
edge labelled by $a$.

Notice that the graph $\Gamma_3$ is obtained from the graph
$\Gamma_2$ by removing the edge labelled by $a$ whose initial
vertex has degree 1. Such an edge is called a \emph{hair}, and the
above procedure is used to be called \emph{``cutting hairs''}.


The label of a path $p=e_1e_2 \cdots e_n$ in $\Gamma$, where $e_i
\in E(\Gamma)$, is the word $$lab(p) \equiv lab(e_1)\cdots
lab(e_n) \in (X^{\pm})^*.$$ Notice that the label of the empty
path is the empty word. As usual, we identify the word $lab(p)$
with the corresponding element in $G=gp\langle X | R \rangle$. We
say that $p$ is   a \emph{normal path} (or $p$ is a path in
\emph{normal form}) if $lab(p)$ is a normal word.

If $\Gamma$ is a well-labelled graph then a path $p$ in $\Gamma$
is freely reduced if and only if $lab(p)$ is a freely reduced
word.
Otherwise $p$  can be converted into a freely reduced path $p'$ by
iteratively removing of the subpaths   $e\overline{e}$
(\emph{backtrackings}) (\cite{mar_meak, kap-m}).  Thus
$$\iota(p')=\iota(p), \ \tau(p')=\tau(p) \ \; {\rm and } \ \; lab(p)=_{FG(X)} lab(p'),$$ where \fbox{$FG(X)$} is a free group
with a free basis $X$. We say that $p'$ is obtained from $p$ by
\emph{free reductions}.


If $v_1,v_2 \in V(\Gamma)$ and $p$ is a path in $\Gamma$ such that
$$\iota(p)=v_1, \ \tau(p)=v_2 \ {\rm and } \ lab(p)\equiv u,$$
then, following the automata theoretic notation, we simply write
\fbox{$v_1 \cdot u=v_2$} to summarize this situation, and  say
that the word $u$ is \emph{readable} at $v_1$ in $\Gamma$.

A pair \fbox{$(\Gamma, v_0)$}  consisting  of the graph $\Gamma$
and the \emph{basepoint} $v_0$ (a distinguished vertex of the
graph $\Gamma$) is called  a \emph{pointed graph}.

Following the notation of Gitik (\cite{gi_sep}) we denote the set
of all closed paths in $\Gamma$ starting at $v_0$  by
\fbox{$Loop(\Gamma, v_0)$},  and the image of $lab(Loop(\Gamma,
v_0))$ in $G=gp\langle X | R \rangle$  by \fbox{$Lab(\Gamma,
v_0)$}. More precisely,
$$Loop(\Gamma, v_0)=\{ p \; | \; p {\rm  \ is \ a \ path \ in \ \Gamma \
with} \ \iota(p)=\tau(p)=v_0\}, $$
$$Lab(\Gamma,v_0)=\{g \in G \; | \; \exists p \in Loop(\Gamma,
v_0) \; : \; lab(p)=_G g \}.$$


It is easy to see that $Lab(\Gamma, v_0)$ is a subgroup of $G$
(\cite{gi_sep}). Moreover, $Lab(\Gamma,v)=gLab(\Gamma,u)g^{-1}$,
where $g=_G lab(p)$, and $p$ is a path in $\Gamma$ from $v$ to $u$
(\cite{kap-m}).
If $V(\Gamma)=\{v_0\}$ and $E(\Gamma)=\emptyset$ then we assume
that $H=\{1\}$.

We say that $H=Lab(\Gamma, v_0)$ is \emph{the subgroup of $G$
determined by the graph $(\Gamma,v_0)$}.  Thus any pointed graph
labelled by $X^{\pm}$, where $X$ is a generating set of a group
$G$, determines a subgroup of $G$. This argues the use of the name
\emph{subgroup graphs} for such graphs.

\subsection*{Morphisms of Labelled Graphs} \label{sec:Morphisms
Of Well-Labelled Graphs}

Let $\Gamma$ and $\Delta$ be graphs labelled with $X^{\pm}$. The
map $\pi:\Gamma \rightarrow \Delta$ is called a \emph{morphism of
labelled graphs}, if $\pi$ takes vertices to vertices, edges to
edges, preserves labels of direct edges and has the property that
$$ \iota(\pi(e))=\pi(\iota(e)) \ {\rm and } \
\tau(\pi(e))=\pi(\tau(e)), \ \forall e\in E(\Gamma).$$
An injective morphism of labelled graphs is called an
\emph{embedding}. If $\pi$ is an embedding then we say that the
graph $\Gamma$ \emph{embeds} in the graph $\Delta$.


A \emph{morphism of pointed labelled graphs} $\pi:(\Gamma_1,v_1)
\rightarrow (\Gamma_2,v_2)$  is a morphism of underlying labelled
graphs $ \pi: \Gamma_1\rightarrow \Gamma_2$ which preserves the
basepoint $\pi(v_1)=v_2$. If $\Gamma_2$ is well-labelled then
there exists at most one such morphism (\cite{kap-m}).


\begin{remark}[\cite{kap-m}] \label{unique isomorphism}
{\rm  If two pointed well-labelled (with $X^{\pm}$) graphs
$(\Gamma_1,v_1)$ and $(\Gamma_2,v_2)$  are isomorphic, then there
exists a unique isomorphism $\pi:(\Gamma_1,v_1) \rightarrow
(\Gamma_2,v_2)$. Therefore $(\Gamma_1,v_1)$ and $(\Gamma_2,v_2)$
can be identified via $\pi$. In this case we sometimes write
$(\Gamma_1,v_1)=(\Gamma_2,v_2)$.} \e
\end{remark}

The notation $\Gamma_1=\Gamma_2$ means that there exists an
isomorphism between these two graphs. More precisely, one can find
$v_i \in V(\Gamma_i)$ ($i \in \{1,2\}$) such that
$(\Gamma_1,v_1)=(\Gamma_2,v_2)$ in the sense of Remark~\ref{unique
isomorphism}.


\section{Subgroup Graphs}
\label{sec:SubgroupGraphs}

The current section is devoted to the discussion on subgroup
graphs constructed by the generalized Stallings' folding
algorithm. The main results of \cite{m-foldings} concerning these
graphs, which are essential for the present paper, are summarized
in terms of free products in Theorem~\ref{thm: properties of
subgroup graphs} below. The notion of reduced precover is
explained right after the theorem along the rest of this section.


\begin{thm} \label{thm: properties of subgroup graphs}
Let $H=\langle h_1, \cdots, h_k \rangle$ be a finitely generated
subgroup of a  free product of finite groups $ G=G_1 \ast  G_2$.

Then there is an algorithm (\underline{the generalized Stallings'
folding  algorithm}) which  constructs a finite labelled graph
$(\Gamma(H),v_0)$ with the following properties:
\begin{itemize}
\item[(1)] $ {Lab(\Gamma(H),v_0)}= {H}. $

\item[(2)] Up to isomorphism, $(\Gamma(H),v_0)$ is  a unique
\underline{reduced precover} of $G$ determining $H$.

\item[(3)] Let $m$  be the sum of the lengths of words $h_1,
\ldots h_n$. Then the algorithm computes $(\Gamma(H),v_0)$ in time
$O(m^2)$.
Moreover, $|V(\Gamma(H))|$ and  $|E(\Gamma(H))|$ are proportional
to $m$.

\end{itemize}
\end{thm}

Throughout the present paper the notation \fbox{$(\Gamma(H),v_0)$}
is always used for the finite labelled graph  constructed by the
generalized Stallings' folding  algorithm for a finitely generated
subgroup $H$ of a free product of finite groups $G=G_1 \ast G_2$.


\subsection*{Precovers} Roughly speaking,
\emph{precovers} are subgroup graphs, corresponding to subgroups
of amalgamated products, with a very particular structure. This
notion was defined by Gitik in \cite{gi_sep} and actively employed
by the author in \cite{m-foldings, m-algI, m-algII}. Below we
define precovers in term of free products and recall some of their
properties which are essential to the present paper.

Let $\Gamma$ be a graph well-labelled with $X^{\pm}$, where $X=X_1
\cup X_2$ is the generating set of $G=G_1 \ast  G_2$  given by
($\ast$) and ($\ast \ast$).
We view $\Gamma$ as a two colored graph: one color for each one of
the generating sets $X_1$ and $X_2$ of the factors $G_1$ and
$G_2$, respectively.

The vertex $v \in V(\Gamma)$ is called \emph{$X_i$-monochromatic}
if all the edges of $\Gamma$ incident with $v$ are labelled with
$X_i^{\pm}$, for some $i \in \{1,2\}$. We denote the set of
$X_i$-monochromatic vertices of $\Gamma$ by $VM_i(\Gamma)$ and put
$VM(\Gamma)= VM_1(\Gamma) \cup VM_2(\Gamma)$.

We say that a vertex $v \in V(\Gamma)$ is \emph{bichromatic} if
there exist edges $e_1$ and $e_2$ in $\Gamma$ with
$$\iota(e_1)=\iota(e_2)=v \ {\rm and} \ lab(e_i) \in X_i^{\pm}, \ i \in \{1,2\}.$$
The  set of bichromatic vertices of $\Gamma$ is denoted by
$VB(\Gamma)$.

A subgraph of $\Gamma$ is called \emph{monochromatic} if it is
labelled only with $X_1^{\pm}$ or only with $X_2^{\pm}$. An
\emph{$X_i$-monochromatic component} of $\Gamma$ ($i \in \{1,2\}$)
is a maximal connected subgraph of $\Gamma$ labelled with
$X_i^{\pm}$, which contains at least one edge.
Thus monochromatic components of $\Gamma$ are graphs determining
subgroups of the factors, $G_1$ or $G_2$.

We say that a graph $\Gamma$ is \emph{$ G$-based} if any path $p
\subseteq \Gamma$ with $lab(p)=_G 1$ is closed. Thus if $\Gamma$
is $G$-based then, obviously, it is well-labelled with $X^{\pm}$.

\begin{defin}[Definition of Precover] A $G$-based  graph $\Gamma$
is a \underline{precover} of $G=G_1 \ast G_2$ if each
$X_i$-monochromatic component of $\Gamma$ is a \underline{cover}
of $G_i$  ($i \in \{1,2\}$).
\end{defin}

Following the terminology of Gitik (\cite{gi_sep}), we use the
term \emph{``covers of $G$''} for \emph{relative (coset) Cayley
graphs} of $G$ and denote by \fbox{$Cayley(G,S)$} the coset Cayley
graph of $G$ relative to the subgroup $S$ of
$G$.\footnote{Whenever the notation $Cayley(G,S)$ is used, it
always means that $S$ is a subgroup of the group $G$ and the
presentation of $G$ is fixed and clear from the context. }
If $S=\{1\}$, then $Cayley(G,S)$ is the \emph{Cayley graph} of $G$
and the notation \fbox{$Cayley(G)$} is used.

Note that the use of the term ``covers'' is adjusted by the  well
known fact that a geometric realization of a coset Cayley graph of
$G$ relative to some $S \leq G$ is a 1-skeleton of a topological
cover corresponding to $S$ of the standard 2-complex representing
the group $G$ (see \cite{stil}, pp.162-163).

\begin{remark}
{\rm Recall that $G=G_1 \ast G_2=gp\langle X|R \rangle$ is given
by $(\ast)$ and $(\ast \ast)$.

Let $\Gamma$ be a graph well-labelled with $X^{\pm}$ such that
each $X_i$-monochromatic component of $\Gamma$ is a cover of $G_i$
($i \in \{1,2\}$). Hence $\Gamma$ is $G$-based, because each cover
of $G_i$ is a $G_i$-based graph.

This allows one to simplify the definition of precovers in the
case of free products by saying that \emph{a graph $\Gamma$ is a
\underline{precover} of $G=G_1 \ast G_2$ if each
$X_i$-monochromatic component of $\Gamma$ is a cover of $G_i$ ($i
\in \{1,2\}$)}.}

 \e
\end{remark}

\begin{conv}
By the above definition, a precover doesn't have to be a connected
graph. However along this paper we restrict our attention only to
connected precovers. Thus any time this term
 is used, we always mean that the corresponding graph
is connected unless it is stated otherwise.

We follow the convention that a graph $\Gamma$ with
$V(\Gamma)=\{v\}$ and $E(\Gamma)=\emptyset$ determining the
trivial subgroup (that is $Lab(\Gamma,v)=\{1\}$) is a (an empty)
precover of $G$.  \e
\end{conv}


\begin{ex} \label{ex:Precovers}
{\rm
 Let $G=\mathbb{Z}_4 \ast  \mathbb{Z}_6=gp\langle x,y | x^4, y^6  \rangle$.

The graph  $\Gamma_1$ on Figure \ref{fig:Precovers} is an example
of a precover of $G$ with one monochromatic component.
$\Gamma_2$, $\Gamma_4$ are examples of precovers of $G$ with two
monochromatic components.

The graph $\Gamma_3$ is not a precover of $G$ because its
$\{x\}$-monochromatic components are not covers of  $\mathbb{Z}_4
$. }\e
\end{ex}
\begin{figure}[!h]
\psfrag{x }[][]{$x$} \psfrag{y }[][]{$y$} \psfrag{v }[][]{$v$}
\psfrag{u }[][]{$u$}
\psfrag{w }[][]{$w$}
\psfrag{x1 - monochromatic vertex }[][]{{\footnotesize
$\{x\}$-monochromatic vertex}}
\psfrag{y1 - monochromatic vertex }[][]{\footnotesize
{$\{y\}$-monochromatic vertex}}
\psfrag{ bichromatic vertex }[][]{\footnotesize {bichromatic
vertex}}
\psfragscanon \psfrag{G }[][]{{\Large $\Gamma_1$}}
\psfragscanon \psfrag{K }[][]{{\Large $\Gamma_2$}}
\psfragscanon \psfrag{L }[][]{{\Large $\Gamma_3$}}
\psfragscanon \psfrag{A }[][]{{\Large $\Gamma_4$}}
\includegraphics[width=\textwidth]{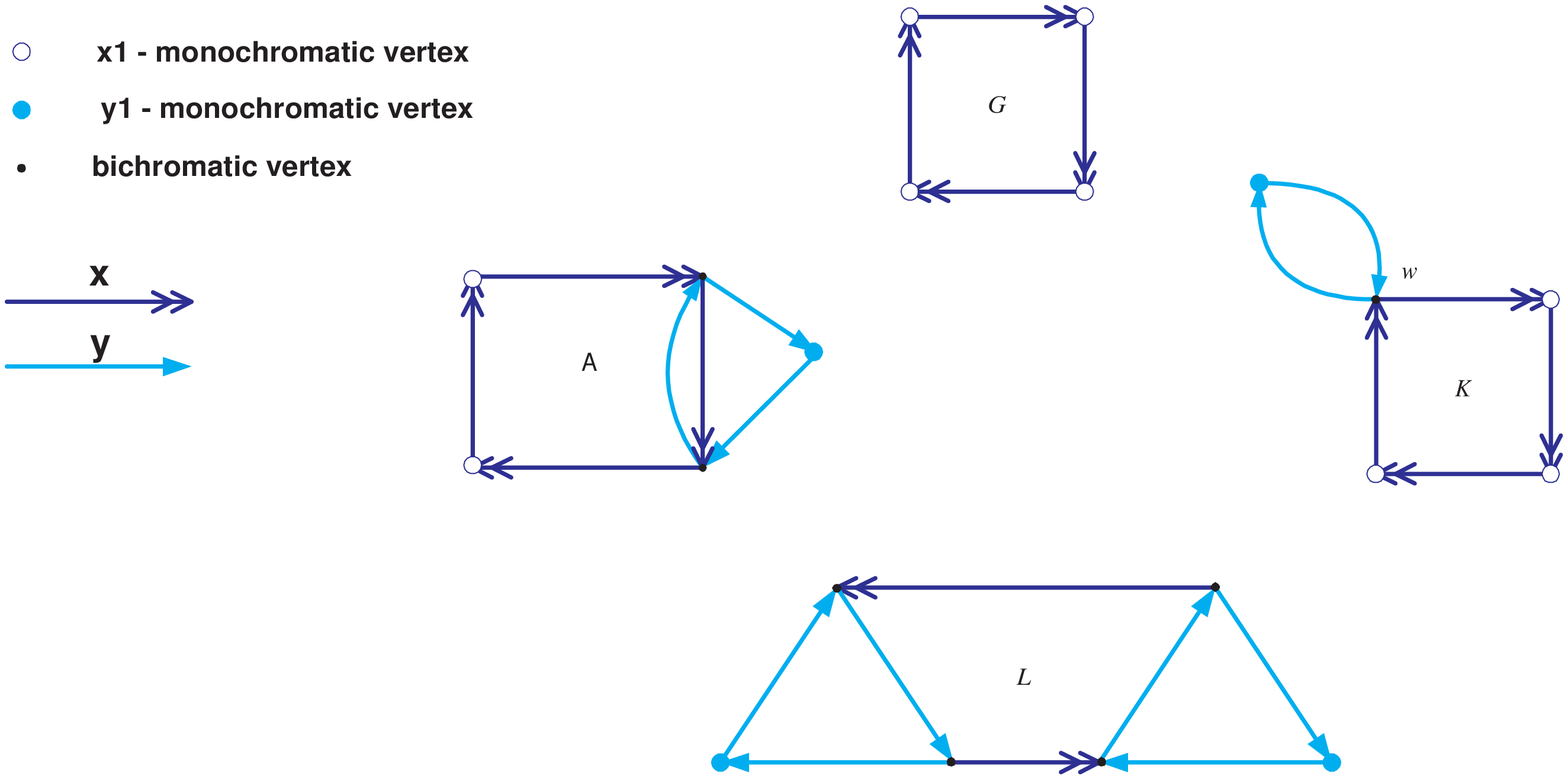}
\caption{ \label{fig:Precovers}}
\end{figure}


A graph $\Gamma$ is \emph{$x$-saturated} at $v \in V(\Gamma)$, if
there exists $e \in E(\Gamma)$ with $\iota(e)=v$ and $lab(e)=x$
($x \in X$). $\Gamma$ is \emph{$X^{\pm}$-saturated} if it is
$x$-saturated for each $x \in X^{\pm}$ at each  $v \in V(\Gamma)$.

\begin{lem}[Lemma 1.5 in \cite{gi_sep}] \label{lemma1.5}
Let $G=gp\langle X|R \rangle$ be a group and let $(\Gamma,v_0)$ be
a graph well-labelled with $X^{\pm}$. Denote $Lab(\Gamma,v_0)=S$.
Then
\begin{itemize}
    \item $\Gamma$ is $G$-based if and only if it can be embedded in $(Cayley(G,S), S~\cdot~1)$,
    \item $\Gamma$ is $G$-based and $X^{\pm}$-saturated if and only if it is isomorphic to  \linebreak[4] $(Cayley(G,S), S
    \cdot~1)$.~
    \footnote{We write $S \cdot 1$ instead of the usual $S1=S$ to distinguish this vertex of $Cayley(G,S)$ as the basepoint of the
    graph.}
\end{itemize}
\end{lem}

\begin{cor}
If $\Gamma$ is a precover of $G$ with $Lab(\Gamma,v_0)=H \leq G$
then  $\Gamma$ is a subgraph of $Cayley(G,H)$.
\end{cor}

Thus a precover of $G$ can be viewed as a part of the
corresponding cover  of $G$, which explains the use of the term
``precovers''.

\begin{defin}[Definition of Reduced Precover]
A \emph{reduced precover} of $G$ is a precover $(\Gamma,v_0)$ of
$G$ with no \underline{redundant monochromatic components}.

A $X_i$-monochromatic component $C$ of the precover $(\Gamma,v_0)$
is \underline{redundant} if the following holds
\begin{itemize}
\item $Lab(C,v)=\{1\}$ (equivalently, by Lemma \ref{lemma1.5},
$C=Cayley(G_i)$),
\item $|VB(C)| \leq 1$,
\item $v_0 \not\in VM(C)$.
\end{itemize}
\end{defin}

\begin{ex}
{\rm
 Let $G=\mathbb{Z}_4 \ast  \mathbb{Z}_6=gp\langle x,y | x^4, y^6  \rangle$.

Any choice of a basepoint in the graph $\Gamma_1$ on Figure
\ref{fig:Precovers} yields a non reduced precover, while any
basepoint of  $\Gamma_4$ gives a reduced precover.

In the graph $\Gamma_2$ any choice of the basepoint $v$ except
that of $w$ (that is $v=w$) makes $(\Gamma_2,v)$ to be a reduced
precover of $G$.} \e
\end{ex}

\begin{remark}[\cite{m-foldings}] \label{remark: morphism of precovers}
{\rm Let $\phi: \Gamma \rightarrow \Delta$ be a morphism of
labelled graphs. If $\Gamma$ is a precover of $G$, then
$\phi(\Gamma)$ is a precover of $G$ as well. }\e
\end{remark}



\section{The Basic Step}
\label{subsection: TheBasicStep}

Let $G=G_1 \ast G_2$ be a free product of finite groups given by
$(\ast)$ and $(\ast \ast)$.

Let $(\Gamma,v_0)$ is a  finite pointed $G$-based graph with
$Lab(\Gamma,v_0)=H \leq G$.

Let $C$ be a $X_i$-monochromatic  component of $\Gamma$ which is a
cover of $G_i$ ($i \in \{1,2\}$). Let $v \in V(C)$ be the
basepoint of $C$. Let $T(C)$ \label{construction of span tree} be
a spanning tree of $C$ with the root vertex $v$.

Let $P_v$ be  an \emph{approach path} in $\Gamma$ from the
basepoint $v_0$ to a vertex $v \in V(C)$ (we assume that $P_v$ is
freely reduced). We put $g_v \equiv lab(P_v)$.

Let $P_v=P_{v1} \cdots P_{vm}$ be a decomposition of $P_v$ into
maximal monochromatic paths. Without loss of generality, we can
assume that $P_{vm} \cap C=\{v\}$. Otherwise, we choose the
basepoint of $C$ to be $v'=\tau(P_{v(m-1)})=\iota(P_{vm})$ and
take the approach path $P_{v'}$ to be $P_{v'}=P_{v1} \cdots
P_{v(m-1)}$.

Following the above assumption, whenever $v_0 \in V(C)$ we chose
$v=v_0$. Thus the path $P_v$ is empty and $g_{v} =_G 1$.

\begin{figure}[!htb]
\begin{center}
\psfrag{B }[][]{{\large $\Gamma(H)$}} \psfrag{A }[][]{{\large
$\Gamma'$}}
\psfrag{u1 }[][]{$v$} \psfrag{v0 }[][]{$v_0$} \psfrag{C }[][]{$C$}
\psfrag{p }[][]{$P_v$}

\includegraphics[width=0.6\textwidth]{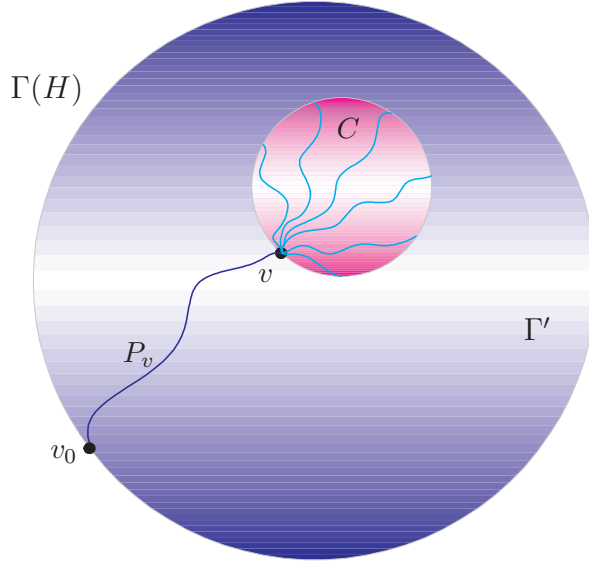}
\caption[A basic step in our computation of  a Kurosh
Decomposition]{ \footnotesize The collection of bright paths
correspond to the spanning tree $T(C)$.
 \label{ex:KuroshDecomp}}
\end{center}
\end{figure}

Let $\Gamma'$ be the graph obtained from $\Gamma$ by removing all
the edges of $C$ which are not in $E(T(C))$. More precisely,
$$E(\Gamma')=E(\Gamma) \setminus E(T(C)), \ \
V(\Gamma')=V(\Gamma).$$

Evidently, the graph $\Gamma'$ is connected.
Roughly speaking, it is a subgraph of $\Gamma$ with $v_0 \in
V(\Gamma')$. Hence $(\Gamma',v_0)$ is a finite pointed $G$-based
graph. Moreover,
\begin{equation} \label{eq: intersection of Gamma' C}
\Gamma' \cap C=T(C)
\end{equation}
 Thus
\begin{equation} \label{eq: empty intersection of loops Gamma' C}
Loop(C,v) \cap Loop(\Gamma',v) =\emptyset.
\end{equation}

{ \ }

To exploit the connection between $Lab(\Gamma,v_0)$, $Lab(C,v)$
and $Lab(\Gamma',v_0)$ we need the following classical result.

\begin{lem}[Lemma IV.1.7 \cite{l_s}] \label{free product}
Let $A$, $B$ be subgroups of a group $G$ such that $A \cup B$
generates $G$, $A \cap B=\{1\}$, and if $g_1, \ldots , g_n$ is a
\underline{reduced sequence} with $n>0$ (that is each $g_i$ is in
one of $A$ or $B$ and successive $g_i$, $g_{i+1}$ are not in the
same factor),  then $g_1g_2 \ldots g_n \neq_G 1$. Then $G \simeq A
\ast B$.
\end{lem}

Now we are ready to give the desired connection. The following
lemma is stated in terms of the above notation.


\begin{lem} \label{H is free product of}
The following holds.
\begin{enumerate}
 \item [(i)] $H=\langle g_{v} Lab(C,v)g_{v}^{-1}, Lab(\Gamma',v_0) \rangle.$
 \item [(ii)]
$\langle g_{v} Lab(C,v)g_{v}^{-1}, Lab(\Gamma',v_0) \rangle=g_{v}
Lab(C,v_r)g_{v}^{-1} \ast Lab(\Gamma',v_0).$
 \item [(iii)] $H=g_{v} Lab(C,v)g_{v}^{-1} \ast Lab(\Gamma',v_0).$
\end{enumerate}
\end{lem}

\begin{proof}
\textbf{(i)}  \
%

Since $Lab(P_{v}Loop(C,v)\overline{P_{v}})=_G g_{v}
Lab(C,v)g_{v}^{-1}$  and
$$P_{v}Loop(C,v)\overline{P_{v}} \subseteq
Loop(\Gamma,v_0),$$ we have $g_{v} Lab(C,v)g_{v}^{-1} \leq H$.

On the other hand, $(\Gamma',v_0)$ embeds in $(\Gamma,v_0)$. Hence
$Lab(\Gamma',v_0) \leq Lab(\Gamma,v_0)=H$. Therefore
\begin{equation} \label{eq:LabInH}
\langle g_{v} Lab(C,v)g_{v}^{-1}, Lab(\Gamma',v_0) \rangle
\subseteq H.
\end{equation}

Conversely, let $h \in H$. Thus  there exists a path $q$ in
$\Gamma$ such that $\iota(q)=\tau(q)=v_0$ and $lab(q) =_G h$.

If $q$ is a path in $\Gamma'$ or in
$P_{v}Loop(C,v)\overline{P_{v}}$. Then we are done.

Otherwise, there is a decomposition $q=q_1t_1q_2t_2 \cdots
t_{k-1}q_{k}$, where $q_i$ are paths in $\Gamma'$ and $t_i$ are
paths in $C$ such that $t_i \cap \Gamma'=\{ \iota(t_i), \tau(t_i)
\}$.

The path $t_i$ can be obtained by the path free reductions from
the path
$$\overline{P_{v}p_{\iota(t_i)}}P_{v}p_{\iota(t_i)}t_i\overline{P_{v}p_{\tau(t_i)}}P_{v}p_{\tau(t_i)},$$
where $p_{\iota(t_i)}$ and $p_{\tau(t_i)}$ are the approach paths
in the spanning tree $T(C)$ from the root vertex $v$ to the
vertices $\iota(t_i)$ and $\tau(t_i)$, respectively. Note that if
$\iota(t_i)=v$ or $\tau(t_i)=v$ then the path $p_{\iota(t_i)}$ or
the path $p_{\tau(t_i)}$, respectively, is empty.

Thus the path $q_{i}t_iq_{i+1}$ can be obtained by the path free
reductions  from the path
$$(q_i\overline{P_{v}p_{\iota(t_i)}}) (P_{v}p_{\iota(t_i)}t_i\overline{P_{v}p_{\tau(t_i)}})
(P_{v}p_{\tau(t_i)}q_{i+1}).$$
 The path
$p_{\iota(t_i)}t_i \overline{p_{\tau(t_i)}}$ is  in $C$ and it is
closed at $v$. Hence the path
$P_{v}p_{\iota(t_i)}t_i\overline{P_{v}p_{\tau(t_i)}}$ is a path
closed at $v_0$ in $\Gamma$ with
$lab(P_{v}p_{\iota(t_i)}t_i\overline{P_{v}p_{\tau(t_i)}}) \in
g_{v}Lab(C,v)g_{v}^{-1}$.

By the construction, the approach paths $P_{v}$, $p_{\iota(t_i)}$
are  in $\Gamma'$. Thus  the paths
$$q_1
\overline{P_{v}p_{\iota(t_{1})}}, \ \ P_{v}p_{\tau(t_{k-1})}q_k \
\ {\rm and} \ \ P_{v}p_{\tau(t_{i-1})}q_i
\overline{P_{v}p_{\iota(t_{i})}} \ \ (\forall \: 2 \leq i \leq
k-1)$$
 are closed at $v_0$ in $\Gamma'$. Hence the labels of these paths are  in $Lab(\Gamma',v_0)$.
Therefore
$$h \equiv lab(q) \in \langle g_{v} Lab(C,v)g_{v}^{-1},
Lab(\Gamma',v_0) \rangle.$$
Thus
\begin{equation}\label{eq:HInLab}
H \subseteq \langle g_{v} Lab(C,v)g_{v}^{-1}, Lab(\Gamma',v_0)
\rangle.
\end{equation}
The combination of (\ref{eq:LabInH}) and (\ref{eq:HInLab}) gives
the desired conclusion that
$$H=\langle g_{v} Lab(C,v)g_{v}^{-1}, Lab(\Gamma',v_0) \rangle.$$

\medskip

%
\textbf{(ii)} \ %
%
We assume that $ Lab(C,v) \neq \{1\}$, otherwise the statement is
trivial. To get the desired equality we have to show that the
conditions of Lemma~\ref{free product} are satisfied.

Since $Lab(P_{v}Loop(C,v)\overline{P_{v}})=_G g_{v}
Lab(C,v)g_{v}^{-1}$ and, by (\ref{eq: empty intersection of loops
Gamma' C}),  $$P_{v}Loop(C,v)\overline{P_{v}} \cap Loop
(\Gamma',v_0)=\emptyset,$$ we have  $g_{v} Lab(C,v)g_{v}^{-1} \cap
Lab(\Gamma',v_0)=\{1\}$.

To prove the satisfaction of the second condition of
Lemma~\ref{free product} we let
$$1 \neq z_l \in g_{v} Lab(C,v)g_{v}^{-1} \ \  { \rm and} \ \  1 \neq w_l
\in Lab(\Gamma',v_0) \ \ (1 < l < k)$$ and show that
$$z_1w_1 \cdots z_kw_k \neq_G 1.$$

Hence there exist closed paths $t_l \in
P_{v}Loop(C,v)\overline{P_{v}}$ and $s_l \in Loop (\Gamma',v_0)$
($1 \leq l \leq k$)  such that
$$ lab(t_l)=_G z_l \ \ {\rm and} \ \ lab(s_l)=_G w_l.$$
Thus $lab(t_l) =_G g_{v} z_l'g_{v}^{-1}$ and there exists a
nonempty path $t_l' \in Loop(C,v)$ such that  $1 \neq z_l' \equiv
lab(t'_l)$ ($1 \leq l \leq k$). Hence $lab(t'_l) \in G_i$ is a
normal word in $G$ of the syllable length 1.

On the other hand, $Lab(\Gamma',v_0)=g_{v} Lab(\Gamma',v)
g_{v}^{-1}$.
Hence, for all $1 \leq l \leq k$, there exists a nonempty path
$s'_l \in Loop (\Gamma',v)$ such that  $$lab(s_l)=_G g_{v}
lab(s'_l)g_{v}^{-1} \ \ (lab(s'_l) \neq 1).$$ Since the graph
$\Gamma'$ is $G$-based,  we can assume (without loss of
generality) that the path $s_l'$ is normal, that is there is a
decomposition of $s_l'$ into maximal monochromatic paths
$s_l'=s'_{l1}s'_{l2} \cdots s'_{lm_l}$ such that
$lab(s_{lf}')\equiv w_{lf} \neq_G 1$, for all $1 \leq f \leq m_l$.
Thus $lab(s'_l)$ is a normal word in $G$ given by the normal
decomposition
$$lab(s'_l) \equiv w_{l1} \cdots w_{lm_l}.$$

We stress that
\begin{eqnarray}
 z_1w_1  & \cdots & z_kw_k   =_G  lab(t_1)lab(s_1) \cdots  lab(t_k)lab(s_k) \nonumber \\
                        & =_G & g_{v} lab(t_1')g_{v}^{-1}g_{v}
                        lab(s'_1)g_{v}^{-1} \cdots g_{v} lab(t_k')g_{v}^{-1}g_{v}
                        lab(s'_k)g_{v}^{-1} \nonumber \\
                        & =_G & g_{v} lab(t'_1)lab(s'_1) \cdots
                        lab(t'_k)lab(s'_k)g_{v}^{-1}  \nonumber
\end{eqnarray}
Note that if $m_l=1$ then, by the construction of $\Gamma'$,
$w_{lm_l} \in G_{\gamma}$ ($1 \leq i \neq \gamma \leq 2$).

If $w_{11}, w_{l1}, w_{(l-1)m_{l-1}} \in G_\gamma$, for all $2
\leq l \leq k$ ($1 \leq i \neq \gamma \leq 2$), then
$$lab(t'_1)lab(s'_1) \cdots lab(t'_k)lab(s'_k)$$ is a normal word in $G$ of syllable
length $k+\sum_{l=1}^k m_l>1$, because $t_l' \in G_i$. Hence
$lab(t'_1)lab(s'_1) \cdots lab(t'_k)lab(s'_k) \neq_G 1$, by the
Normal Form Theorem for Free Products \cite{l_s} (see Section
\ref{sec:Preliminaries}).

Otherwise, $w_{11} \in G_i$  or there exists $2 \leq l \leq k$
such that $w_{l1} \in G_i$ or $w_{(l-1)m_{l-1}} \in G_i$.

Recall that the graph $\Gamma'$ is well-labelled with $X^{\pm}$.
 Since, by our assumption, $C$ is a $X_i$-monochromatic
component of $\Gamma$ which is a cover of $G_i$, each $v \in V(C)$
is $X_i^{\pm}$-saturated. Thus, each path in $\Gamma$ which starts
at such vertex $v$ with label in $G_i$ is a path in $C$. Therefore
 either $s'_{11}$ or $s'_{l1}$ or $s'_{(l-1)m_{l-1}}$ is
in $\Gamma' \cap C=T(C)$.

Let $q \subseteq T(C)$ and $r \in Loop(C,v)$ such that either
$\tau(q)=v$ or $\iota(q)=v$. Thus the paths $qr$ and $rq$,
respectively, are unclosed, because $q$ is unclosed. Since the
graph $\Gamma'$ is $G$-based, we have either $lab(qr) \neq_G 1$ or
$lab(rq) \neq_G 1$.

Moreover, if $q_1, q_2 \subseteq T(C)$ such that
$\tau(q_1)=\iota(q_2)=v$ then the path $q_1rq_2$ is closed if and
only if $q_2=\bar{q_1}$. Thus $q_1rq_2=q_1r\bar{q_1}$. If $lab(r)
\neq_G 1$ then $lab(q_1r\bar{q_1}) \equiv
lab(q_1)lab(r)lab(q_1)^{-1} \neq_G 1$.

Therefore $lab(t'_1)lab(s'_1) \cdots lab(t'_k)lab(s'_k)$ can be
viewed as a normal word in $G$ of length at least $(\sum_{l=1}^k
m_l )-(k-1) > 1.$
Hence $lab(t'_1)lab(s'_1) \cdots lab(t'_k)lab(s'_k) \neq_G 1$, by
the Normal Form Theorem for Free Products \cite{l_s}.
Thus
\begin{eqnarray}
 z_1w_1  & \cdots & z_kw_k   =_G  g_{v_r} lab(t'_1)lab(s'_1) \cdots
                        lab(t'_k)lab(s'_k)g_{v_r}^{-1}  \neq_G 1.  \nonumber
\end{eqnarray}
Therefore the conditions of Lemma~\ref{free product} are
satisfied. Hence $$\langle g_{v_r} Lab(C,v_r)g_{v_r}^{-1},
Lab(\Gamma',v_0) \rangle=g_{v_r} Lab(C,v_r)g_{v_r}^{-1} \ast
Lab(\Gamma',v_0).$$
%
\textbf{(iii)} \ %
%
The  combination of (i) and (ii) yields
$$H= g_{v} Lab(C,v)g_{v}^{-1} \ast Lab(\Gamma',v_0) .$$
\end{proof}


\section{Reading off Kurosh Decompositions}
\label{subsection: The Reading off Kurosh Decomposition}

Let $H$ be a finitely generated subgroup of a free product of
finite groups $G=G_1 \ast G_2$ given by ($\ast$) and ($\ast
\ast$). Consider $\Gamma(H)$ to be the subgroup graph of $H$
constructed by the generalized Stallings algorithm (see Appendix
for the algorithm description).

In the current section we introduce (along with the proof  of
Theorem~\ref{cor: alg_Kurosh}) an algorithm which reads off a
Kurosh decomposition of $H$ from its subgroup graph $\Gamma(H)$.
This algorithm  relays largely on the basic step construction
introduced in the previous section.

Another essential step of the algorithm is provided by
understanding whether the given labelled graph determines a free
subgroup.
In \cite{m-algI} (Theorem 6.4) such a connection was obtained for
subgroup graphs which are reduced precovers. Below we restate this
result in terms of free products of finite groups.

\begin{thm}(Theorem 6.4 in  \cite{m-algI}) \label{thm:
freeness} $H$ is  free if and only if each $X_i$-monochromatic
component of $\Gamma(H)$ is isomorphic to $Cayley(G_i)$, for all
$i \in \{1,2\}$.
\end{thm}

In the case of  free products of finite groups such a connection
can be found even if the given graph is not a precover of $G$.

\begin{lem} \label{lem: last step}
 Let $(\Gamma,v_0)$ be a  finite pointed $G$-based  graph
well-labelled with $X^{\pm}$ such that $Lab(\Gamma,v_0)=H \leq G$.

If all monochromatic components of $\Gamma$ are trees then $H$ if
free.
\end{lem}

To prove this lemma the following technical result from
\cite{m-foldings} is necessary.

\begin{lem} \label{glue_cayley graph to an edge}
Let $(\Gamma,v_0)$ be a  finite pointed graph well-labelled with
$X^{\pm}$ . Let $e$ be an edge of $\Gamma$ with $lab(e) \in
X_i^{\pm}$ ($i \in \{1,2\}$).

Let $(\Delta,u_0)$ be the graph obtained from  $\Gamma$ by gluing
a copy of $Cayley(G_i)$ along the edge $e$, where $u_0$ is the
image of $v_0 $ in $\Delta$.

Then $Lab(\Gamma,v_0)=Lab(\Delta,u_0)$.
\end{lem}

\begin{proof}[Proof of Lemma \ref{lem: last step}] \
By Lemma \ref{lemma1.5}, any finite well-labelled
$X_i$-monochromatic tree embeds into $Cayley(G_i)$ ($i \in
\{1,2\}$). Thus  the graph $(\Gamma,v_0)$ embeds into the graph
$(\Gamma',v'_0)$ obtained  by gluing copies of $Cayley(G_i)$ to
each $X_i$-monochromatic tree of $\Gamma$ ($v_0'$ is the inherited
base point). Moreover, the resulting graph $(\Gamma',v'_0)$ is a
precover of $G$.

By Lemma~\ref{glue_cayley graph to an edge},
$Lab(\Gamma',v'_0)=Lab(\Gamma,v_0)=H$. If $\Gamma'$ is not a
reduced precover of $G$ then it can be turned to one by removing
redundant components. As is well known from \cite{m-foldings},
this procedure is finite and does not change the determined
subgroup. Therefore, without loss of generality, we assume that
$(\Gamma',v'_0)$ is a reduced precover of $G$.

Hence, by Theorem~\ref{thm: properties of subgroup graphs} (2),
$(\Gamma',v'_0)=(\Gamma(H),u_0)$. Thus, by Theorem~\ref{thm:
freeness}, $H$ is a free group.
\end{proof}


Let $\Gamma$ be a finite $G$-based graph well-labelled with
$X^{\pm}$. We set \fbox{$MCC(\Gamma)$} to be the list of all
Monochromatic Components of $\Gamma$ which are Covers of either
$G_1$ or $G_2$. Since the graph $\Gamma$ is finite, the set
$MCC(\Gamma)$ is finite as well.

%

\begin{thm} \label{cor: alg_Kurosh}
 Let $h_1, \ldots, h_n \in G$. Then there exists an
algorithm which computes a Kurosh decomposition of the subgroup
$H=\langle h_1, \ldots, h_n \rangle \leq G.$
\end{thm}
\begin{proof}
First we construct the subgroup graph $(\Gamma(H),v_0)$ using the
generalized Stallings algorithm (see the Appendix).

Then we iteratively apply the basic step construction described in
Section~\ref{subsection: TheBasicStep} to the monochromatic
components of $\Gamma(H)$.  Since $k=|MCC(\Gamma(H))| < \infty$
this process is finite. We start from a monochromatic component
$C_0$ of $\Gamma(H)$ such that $v_0 \in V(C_0)$. We take $v_0$ as
the basepoint of $C_0$ and let the approach path be empty. This
yields the graph $\Gamma'_1$ with $MCC(\Gamma_1')=MCC(\Gamma(H))
\setminus \{C_0\}$.

Let $\Gamma'_i$ be the graph obtained after $(i-1)$ consequence
applications of the basic step to the graphs $\Gamma(H),
\Gamma'_1, \ldots \Gamma'_{i-1}$ and the monochromatic components
$(C_0,v_0), (C_1,v_1), \ldots, (C_{i-1},v_{i-1})$, respectively.
Thus  $MCC(\Gamma_i')=MCC(\Gamma(H)) \setminus \{C_0, (C_1,v_1),
\ldots, (C_{i-1},v_{i-1})\}$.

Our next application of the basic step is to the graph $\Gamma'_i$
and a monochromatic component $C_i \in MCC(\Gamma'_i)$ such that
$VB(C_{i-1}) \cap VB(C_i) \neq \emptyset$. We pick a vertex $v_i
\in VB(C_{i-1}) \cap VB(C_i)$ to be the base point of $C_i$ and
choose the appropriate approach path $P_{v_i}$.

After $k=|MCC(\Gamma(H))|$ steps this process gives a finite graph
$(\Delta,v_0)$ whose monochromatic components are trees, that is
$MCC(\Delta)=\emptyset$ and $Lab(\Delta,v_0)$ is a free group, by
Lemma~\ref{lem: last step}.

Lemma~\ref{H is free product of} yields the following Kurosh
decomposition of $H$.
$$H=\big( \ast_{0 \leq i \leq (k-1)} lab(P_{v_i})Lab(C_i, v_i)lab(P_{v_i})^{-1}
\big) \ast Lab(\Delta,v_0),$$ where $F=Lab(\Delta,v_0)$ is a free
group.

Since the factors $G_1$ and $G_2$ are finite as well as all the
monochromatic components $C_i$ ($0 \leq i \leq k-1$), which are
their covers, it is possible to compute $Lab(C_i, v_i)$ applying,
for instance, the well-known Reidemeister-Schreier procedure
(p.102 in \cite{l_s}).

In order to find a free basis $S$ of $F=Lab(\Delta,v_0)$, we
proceed according to the well-known algorithm for subgroups of
free groups \cite{kap-m, mar_meak, stal} which computes a free
basis defined by a labelled graph. Thus
$$S=\{lab(p_{\iota(e)}e \overline{p_{\tau(e)}}) \; | \; e \in E(\Delta)^+ \setminus
T(\Delta)\},$$ where $T(\Delta)$ is a  spanning tree  of $\Delta$,
and  $p_v$ is the unique freely reduced path in $T$ with
$\iota(p_v)=v_0$ and $\tau(p_v)=v$.


Thus $lab_{FG(X)}(\Delta,v_0)=FG(S)$, while
$Lab(\Delta,v_0)=FG(S)/FG(S) \cap N$, where $N$ is the normal
closure of $R$ in $FG(X)$.

However $FG(S) \cap N=\{1\}$. Indeed, let $1 \neq w \in FG(S) \cap
N$. Without loss of generality we can assume that $w$ is a freely
reduced word.

Thus there exists a reduced path $p$ in $(\Delta,v_0)$ closed at
$v_0$ with $\iota(p)=\tau(p)=v_0$ and $lab(p) \equiv w$. Let
$p=p_1 \cdots p_m$ be a decomposition of $p$ into maximal
monochromatic paths. By the construction of $(\Delta,v_0)$, all
its monochromatic components are trees, therefore all the paths
$p_i$ ($1 \leq i \leq m$) are unclosed and hence $lab(p_i) \neq _G
1$. Thus $lab(p) \equiv lab(p_1) \cdots lab(p_m)$ is a normal word
in $G$. Therefore, by the Normal Form Theorem for Free Products,
$w \equiv  lab(p) \neq_G 1$, that is $w \not\in N$. Thus
$Lab(\Delta,v_0)=FG(S)$.

Hence $$H=\big( \ast_{1 \leq j \leq m} g_jH_jg_j^{-1} \big) \ast
FG(S),$$
where $H_j=Lab(C_i, v_i) \neq \{1\}$ and $g_j \equiv
lab(P_{v_i})$.

\end{proof}

\begin{remark} \label{cor: group presentation}
{\rm As an immediate consequence of the above computation  the
group presentation of $H$ is obtained even if $[G:H]=\infty$ and
the Reidemeister-Schreier process doesn't work.

Indeed, since the subgroups $H_j$ have  finite index in the free
factors of $G$, their group presentation $H_j=gp\langle Y_j \; |
\; R_j \rangle$ as a subgroup of a free factor can be  computed
using Reidemeister-Schreier process. Thus
$$H=gp\langle S, \ g_jY_jg_j^{-1} \; | \;
g_jR_jg_j^{-1}\rangle.$$}
\e

\end{remark}


\subsection*{Complexity Issues}

 It should be stressed that in contrast with papers that establish
the exploration of the algorithms complexity as their primary goal
(see, for instance, \cite{generic-case, average-case, tuikan}), we
do it rapidly (sketchy) viewing in its analysis a way to emphasize
the effectiveness  of our graph theoretical approach.

The main purpose of the complexity analysis below is to estimate
our graph theoretical methods applied to read off a Kurosh
decomposition of a subgroup from its subgroup graph.

To this end we assume that the free product of finite groups
$G=G_1 \ast G_2$ is given via ($\ast$) and ($\ast \ast$),
respectively, and that this presentation is not a part of the
input. We assume as well that the Cayley graphs and all the
relative Cayley graphs of the free factors $G_1$ and $G_2$ are
given for ``free'' (see the Appendix for the discussion on given
data and input).  These assumptions allow us to be concentrated
only on the estimation of the algorithm presented along with the
proof of Theorem~\ref{cor: alg_Kurosh}.

Indeed, if the group presentations of the free factors $G_1$ and
$G_2$ are a part of the input (the \emph{uniform version} of the
algorithm) then we have to build the groups $G_1$ and $G_2$ (that
is to construct their Cayley graphs and relative Cayley graphs).

Since the groups $G_1$ and $G_2$ are finite,  the Todd-Coxeter
algorithm and the Knuth Bendix algorithm are suitable \cite{l_s}
for these purposes. Then the complexity of the construction
depends on the group presentation of $G_1$ and $G_2$ we have: it
could be even exponential in the size of the presentation
\cite{cdhw73}. Therefore the above algorithm  with these
additional constructions could take time exponential in the size
of the input.
%

\subsection*{Complexity Analysis }

By Theorem~\ref{thm: properties of subgroup graphs} $(3)$,  the
construction of $\Gamma(H)$ takes $O(m^2)$, where $m$ is the sum
of lengths of the input subgroup generators $h_1, \ldots, h_n$.

The detecting of monochromatic components in the constructed graph
takes $ \: O(|E(\Gamma(H))|) \: $, that is  $O(m)$. Since all the
essential information about  $G_1$ and  $G_2$ is given and it is
not a part of the input, verifications concerning a particular
monochromatic component of $\Gamma(H)$,  takes $O(1)$.

Since the construction of a spanning tree in a monochromatic
component $C$ of $\Gamma(H)$ takes $ \: O(|E(C)|) \: $, this
procedure applied to all monochromatic components of $\Gamma(H)$
takes $ \: O(|E(\Gamma(H))|) \: $. Therefore to construct the
graph $\Delta$ from $\Gamma(H)$ takes $ \: O(|E(\Gamma(H))|) \: $,
that is $O(m)$.

The construction of the free basis of $F=Lab(\Delta,v_0)$ in the
described way takes $O(|E(\Delta)|^2)$, by \cite{b-m-m-w}. Since
$|E(\Delta)|<|E(\Gamma(H))|$, the above construction takes $ \:
O(|E(\Gamma(H))|^2) \: $, that is  $O(m^2)$.

Therefore the complexity of the algorithm given along with the
proof of Corollary \ref{cor: alg_Kurosh} equals $O(m^2)$.

If the subgroup $H$ is given by the graph $(\Gamma(H),v_0)$ and
not by a finite set of subgroup generators, then the complexity is
$O(|E(\Gamma(H))|^2)$. Thus in both cases the algorithm is
quadratic in the size of the input.


\begin{figure}[!h]
\begin{center}
\psfragscanon
\psfrag{A0 }[][]{{$\Gamma(H)$}}
\psfrag{A1 }[][]{{$\Gamma'_1$}} \psfrag{A2 }[][]{{$\Gamma'_2$}}
\psfrag{A3 }[][]{{$\Gamma'_3$}} \psfrag{A4 }[][]{{$\Gamma'_4$}}
\psfrag{A5 }[][]{{$\Delta$}}
\psfrag{x }[][]{$a$} \psfrag{y }[][]{$b$}

\psfrag{v0 }[][]{\small $v_0$}

\psfrag{v1 }[][]{\small $v_1$} \psfrag{v2 }[][]{\small $v_2$}
\psfrag{v3 }[][]{\small $v_3$} \psfrag{v4 }[][]{\small $v_4$}
\psfrag{C0 }[][]{\small $C_0$}

\psfrag{C1 }[][]{\small $C_1$} \psfrag{C2 }[][]{\small $C_2$}
\psfrag{C3 }[][]{\small $C_3$} \psfrag{C4 }[][]{\small $C_4$}

\includegraphics[width=0.8\textwidth]{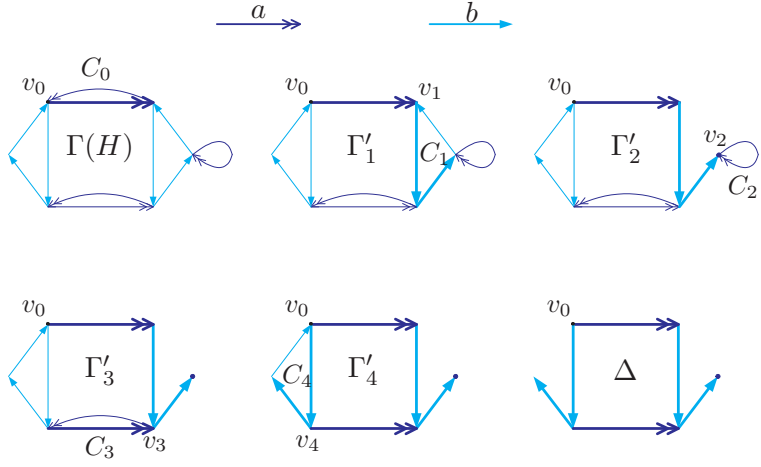}


\caption[An example of our computation of a Kurosh Decomposition]{
\footnotesize A computation of a Kurosh Decomposition of $H$ from
 $\Gamma(H)$. The bold edges correspond to  spanning
 trees of the appropriate monochromatic components.
 \label{Fig: KuroshDecomp1}}
\end{center}
\end{figure}
\begin{ex}
{\rm

Let  $G=Z_2 \ast Z_3=gp\langle a,b \; | \; a^2, b^3\rangle \simeq
PSL_2(Z)$.

Let $H=\langle aba^{-1}b^{-1}, (ba)^3 \rangle \leq G$. We use the
subgroup graph $\Gamma(H)$ constructed by the generalized
Stallings' algorithm (see Example~\ref{ex: free product
construction} and Figure \ref{fig:ConstructionFreePr} for the
precise construction) to read off a Kurosh decomposition of $H$.
The reading procedure described along with the proof of
Theorem~\ref{cor: alg_Kurosh}  is illustrated step by step on
Figure \ref{Fig: KuroshDecomp1}.

The computation of a group presentation of $H$, according to
Corollary~\ref{cor: group presentation}, is presented below.
\begin{eqnarray}
H & = & Lab(\Gamma'_1, v_0) \ast Lab(C_0,v_0) \nonumber \\
& = & Lab(\Gamma'_2, v_0) \ast Lab(C_1,v_1)  \nonumber \\
& = & Lab(\Gamma'_3, v_0) \ast Lab(C_2,v_2)  \nonumber \\
& = & Lab(\Gamma'_4, v_0) \ast Lab(C_3,v_3) \ast (ab^2)\langle a \rangle (ab^2)^{-1} \nonumber \\
& = & Lab(\Delta, v_0) \ast Lab(C_4,v_4) \ast (ab^2)\langle a \rangle (ab^2)^{-1}  \nonumber \\
& = & Lab(\Delta, v_0) \ast Lab(C_4,v_4) \ast (ab^2)\langle a
\rangle (ab^2)^{-1} \nonumber \\
& = & FG(aba^{-1}b{-1})  \ast (ab^2)\langle a \rangle (ab^2)^{-1}.
\nonumber
\end{eqnarray}

Let $e_1=aba^{-1}b^{-1}, \ e_2=(ab^2) a (ab^2)^{-1}$.
Thus $H=gp\langle e_1,e_2 \; | \; e_1, e_2^2 \rangle$.

 } \e
\end{ex}

{ \ }


\appendix
\section{}
\label{subsection:The Adapted Algorithm}

Let $G=G_1 \ast G_2$. Obviously, $G=G_1 \ast_{\{1\}} G_2$. The
assumption that the amalgamated subgroup is trivial simplifies the
algorithm from \cite{m-foldings}, making the  fourth and the sixth
steps to be irrelevant.
Thus the restricted algorithm  takes the following form.

{ \ }

\begin{conv}
{\rm We follow the notation of Grunschlag \cite{grunschlag},
distinguishing between the ``\emph{input}'' and the ``\emph{given
data}'', the information that can be used by the algorithm
\emph{``for free''}, that is it does not affect the complexity
issues.} \e
\end{conv}

{ \ }

\begin{center}
\large{\emph{\underline{\textbf{Algorithm}}}}
\end{center}

\smallskip

\begin{description}
\item[Given] Finite groups $G_1$, $G_2$ and the free product
$G=G_1 \ast G_2$ given via ($\ast$) and ($\ast \ast$),
respectively.

We assume that the Cayley graphs and all the relative Cayley
graphs of the free factors are given.
\item[Input]  A finite set $\{ g_1, \cdots, g_n \} \subseteq G$.
\item[Output] A finite graph $\Gamma(H)$ with a basepoint $v_0$
which is a reduced precover of $G$ and the following holds
\begin{itemize}
 \item
$Lab(\Gamma(H),v_0)=_{G} H$;
 \item $H=\langle g_1, \cdots, g_n \rangle$;
 \item a normal word $w$ is in $H$ if and only if
  there is a loop (at $v_0$) in $\Gamma(H)$
labelled by the word $w$.
 \end{itemize}
\item[Notation] $\Gamma_i$ is the graph obtained after the
execution of the $i$-th step.

%
%
\medskip

    \item[\underline{Step1}] Construct a based set of $n$ loops around a common distinguished
vertex $v_0$, each labelled by a generator of $H$;
    \item[\underline{Step2}] Iteratively fold edges and cut hairs;
 \item[\underline{Step3}] { \ }\\
\texttt{For} { \ } each $X_i$-monochromatic component $C$ of
$\Gamma_2$ ($i=1,2$) { \ } \texttt{Do} \\
\texttt{Begin}\\
    pick an edge $e \in E(C)$; \\
    glue a copy  of $Cayley(G_i)$   on $e$ via identifying $ 1_{G_i} $  with $\iota(e)$ \\
    and identifying the two copies of $e$ in $Cayley(G_i)$ and in $\Gamma_2$; \\
    \texttt{If}  { \ } necessary  { \ } \texttt{Then} { \ } iteratively fold
    edges; \\
 \texttt{End;}

 \item[\underline{Step4}] { \ } \\
%
Reduce  $\Gamma_3$ by iteratively removing all \emph{redundant}
 $X_i$-monochromatic components $C$ which are
\begin{itemize}
 \item $(C,\vartheta)$ is isomorphic to $Cayley(G_i, 1)$;
 \item  $VB(C)=\{\vartheta\}$;
 \item  $v_0  \not\in VM_i(C)$.
\end{itemize}

Let $\Gamma$ be the resulting graph;\\

\texttt{If}  { \ }
$VB(\Gamma)=\emptyset$ and $(\Gamma,v_0)$ is isomorphic to $Cayley(G_i, 1_{G_i})$ \\
\texttt{Then} { \ } we set $V(\Gamma(H))=\{v_0\}$ and
$E(\Gamma(H))=\emptyset$.\\
\texttt{Else} { \ } we set $\Gamma(H)=\Gamma$.

\end{description}

{ \ }


\begin{remark} \label{stal-mar-meak-kap-m}
{\rm The first two steps of the above algorithm correspond
precisely to the Stallings' folding algorithm for finitely
generated subgroups of free groups \cite{stal, mar_meak, kap-m}.}
\e
\end{remark}

\begin{figure}[!htb]
\begin{center}
\psfragscanon \psfrag{A }[][]{{$\Gamma(H)$}}
\psfrag{A1 }[][]{{$\Gamma_1$}} \psfrag{A2 }[][]{{$\Gamma_2$}}
\psfrag{A3 }[][]{{$\Gamma'_3$ }}
\psfrag{a }[][]{\footnotesize $a$} \psfrag{b }[][]{\footnotesize
$b$}

\psfrag{v0 }[][]{\small $v_0$}

\psfrag{x }[][]{$a$} \psfrag{y }[][]{$b$}

\psfrag{v1 }[][]{\small $v_1$} \psfrag{v2 }[][]{\small $v_2$}
\psfrag{v3 }[][]{\small $v_3$} \psfrag{v4 }[][]{\small $v_4$}
\psfrag{C0 }[][]{\small $C_0$}

\psfrag{C1 }[][]{\small $C_1$} \psfrag{C2 }[][]{\small $C_2$}
\psfrag{C3 }[][]{\small $C_3$} \psfrag{C4 }[][]{\small $C_4$}

\includegraphics[width=\textwidth]{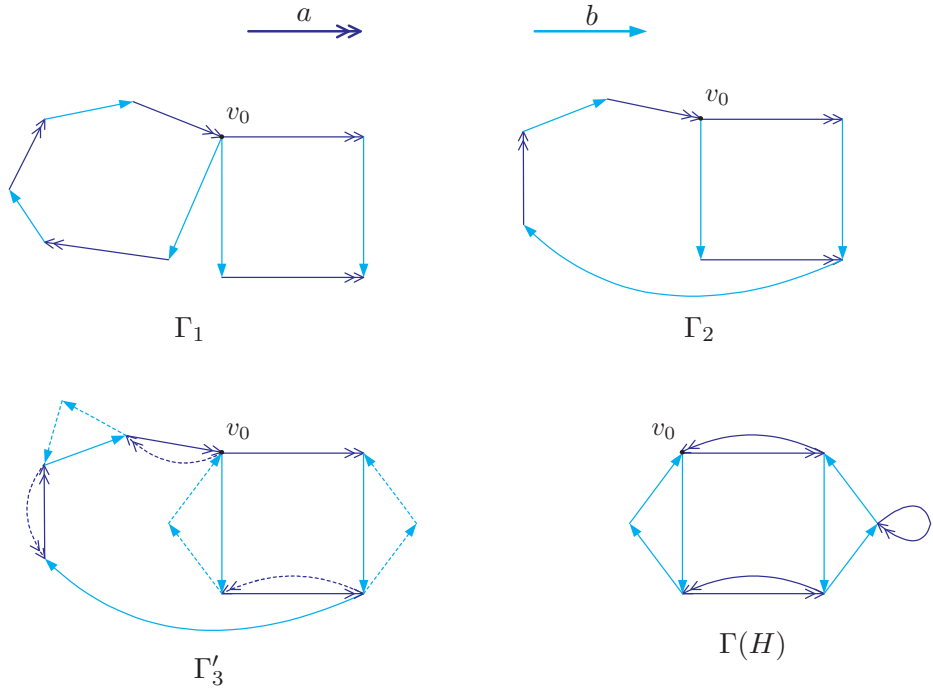}

{ \ } \\

\caption{ \footnotesize The graph $\Gamma'_3$ is an intermediate
graph of the Step 3 obtained after the gluing operations before
the foldings are done. \label{fig:ConstructionFreePr}}
\end{center}
\end{figure}

\begin{ex} \label{ex: free product construction}
{\rm Let  $G=Z_2 \ast Z_3=gp\langle a,b \; | \; a^2, b^3\rangle
\simeq PSL_2(Z)$.

Let $H=\langle aba^{-1}b^{-1}, (ba)^3 \rangle \leq G$. The
construction of $\Gamma(H)$ by the generalized Stallings' folding
algorithm  is presented on Figure \ref{fig:ConstructionFreePr}.}\e

\end{ex}


\end{document}